\documentclass{amsart}

\usepackage{amssymb}

\newtheorem{thm}{Theorem}

\newtheorem{defn}[thm]{Definition}
\newtheorem{lem}[thm]{Lemma}
\newtheorem{cor}[thm]{Corollary}
\newtheorem{conj}[thm]{Conjecture}
\theoremstyle{remark}

\renewcommand{\epsilon}{\varepsilon}
\renewcommand{\geq}{\geqslant}

\title{Relative K-stability of extremal metrics}
\author{J. Stoppa and G. Sz\'ekelyhidi\textsuperscript{\dag}}
\thanks{\dag\,Partially supported by NSF grant DMS-0904223}
\address{Department of Pure Mathematics and Mathematical Statistics, University of Cambridge, Wilberforce Road, Cambridge CB3 0WB, UK}
\email{J.Stoppa@dpmms.cam.ac.uk}
\address{Columbia University, Department of Mathematics, 2990 Broadway
New York, NY 10027, USA}
\email{gabor@math.columbia.edu}
\begin{document}

\begin{abstract}
	We show that if a polarised manifold admits an extremal metric
	then it is K-polystable relative to a maximal torus of
	automorphisms. 
\end{abstract}

\maketitle

\section{Introduction}
Calabi~\cite{Cal82}
introduced the notion of extremal metrics as candidates for
canonical
representatives of K\"ahler classes on compact K\"ahler manifolds.
Unfortunately not all K\"ahler manifolds admit extremal metrics (eg.
Levine~\cite{Lev}) and even if they do, they may not admit them in
all K\"ahler classes (see eg.
Apostolov, Calderbank, Gauduchon, T\o{}nnesen-Friedman~\cite{ACGT3}). 
This makes the question of existence of
extremal metrics quite delicate and there is now a vast literature
on the topic.  We refer to Phong-Sturm \cite{PS} 
for a recent survey and an extensive bibliography.

By definition an extremal metric is a K\"ahler metric whose scalar
curvature has holomorphic gradient vector field. Thus, special cases are
constant scalar curvature K\"ahler (or cscK) metrics and K\"ahler-Einstein
metrics. While one can study these metrics in arbitrary K\"ahler
classes, perhaps the most interesting case is when the K\"ahler class is
the first Chern class of an ample line bundle. Indeed, existence
of a cscK metric on a manifold $M$ in the K\"ahler class $c_1(L)$ for
an ample line bundle $L$, is expected to be closely related to
algebro-geometric properties of the polarised manifold $(M,L)$. This is
expressed by the following.

\begin{conj}[Yau~\cite{Yau93}, Tian~\cite{Tian97}, Donaldson~\cite{Don02}]
	\label{conj:1}
	The manifold $M$ admits a
	cscK metric in the class $c_1(L)$ if and only if the pair
	$(M,L)$ is K-polystable.
\end{conj}

The notion of K-polystability will be recalled below. Building on the
K-semistabi\-lity proved by Donaldson~\cite{Don05} and on the work of
Arezzo-Pacard~\cite{AP06} on blowing up cscK metrics, the first named
author completed the proof of one direction of this conjecture, under
the assumption that the automorphism 
group of $(M,L)$ is discrete.

\begin{thm}[\cite{Sto08} Theorem 1.2] If $M$ admits a cscK metric in $c_1(L)$ and
	$\mathrm{Aut}(M,L)$ is discrete, then $(M,L)$ is K-polystable.
\end{thm}

Using a different approach, this was recently extended to manifolds with
not necessarily discrete automorphism groups by Mabuchi~\cite{Mab1},
\cite{Mab2}.
The aim of the
present paper is to generalise this theorem to the case of extremal
metrics. In this case the conjecture analogous to Conjecture
\ref{conj:1} was formulated by the second named author in \cite{GSz04}.
\begin{conj}\label{conj:2}
 The manifold $M$ admits an extremal metric in the class
	$c_1(L)$ if and only if the pair $(M,L)$ is K-polystable
	relative to a maximal torus of automorphisms of $(M,L)$.
\end{conj}
By generalising the approach in \cite{Sto08} we obtain the following, which is the
main result of this paper.

\begin{thm}\label{thm:main}
	If $M$ admits an extremal metric in $c_1(L)$ then $(M,L)$ is
	K-polystable relative to a maximal torus of automorphisms of
	$(M,L)$. 
\end{thm}

In particular the theorem applies when $M$ admits a cscK metric and has
continuous 
automorphisms, proving that $M$ is K-polystable with respect to all 
test- configurations that commute with a maximal torus of automorphisms,
 but note that this is a priori a weaker condition
than K-polystability (see the next section for the detailed
definitions).

Note that by an example in \cite{ACGT3} relative K-polystability may not be sufficient
to ensure the existence of an extremal metric, so it is likely that 
the conjectures \ref{conj:1} and \ref{conj:2} have to be refined. 

\subsection*{Acknowledgements}
We would like to thank Julius Ross and Richard Thomas for helpful discussions. The second named author would also like to thank D. H. Phong for his encouragement and support. 

\section{Relative K-polystability}
In this section we recall the notion of relative K-polystability
following~\cite{GSz04}. This is a modification of the notion of
K-polystability introduced by Donaldson~\cite{Don02}.

Suppose that $(V,L)$ is a polarised scheme of dimension $n$, with
a $\mathbf{C}^*$ action $\alpha$. 
Let us write $A_k$ for the
infinitesimal generator of the action of $\alpha$  
on $H^0(V,L^k)$, and write $d_k$
for the dimension of $H^0(V,L^k)$. Then $d_k$ is a polynomial of
degree $n$ and $\mathrm{Tr}(A_k)$ is a polynomial
of degree $n+1$ for sufficiently large $k$, so we can write
\[ \begin{aligned}
	d_k &= c_0k^n + c_1k^{n-1} + O(k^{n-2}),\\
	\mathrm{Tr}(A_k) &= a_0k^{n+1} + a_1k^n + O(k^{n-1}),\\
\end{aligned} \]
Donaldson's Futaki invariant is defined to be
\[ F(\alpha) = \frac{c_1}{c_0}a_0 - a_1. \]
Sometimes we will write $F(V,L,\alpha)$ to emphasize the space that
$\alpha$ is acting on.

Suppose in addition that we have a $\mathbf{C}^*$-action 
$\beta$ acting on
$(V,L)$ which commutes with $\alpha$, and write $B_k$ for the
infinitesimal generator of the action on $H^0(V,L^k)$. 
Then $\mathrm{Tr}(A_kB_k)$ is a polynomial of degree $k+2$ for
sufficiently large $k$, and we define the inner product
$\langle\alpha,\beta\rangle$ to be the leading coefficient in the
expansion
\[ \mathrm{Tr}(A_kB_k) - \frac{\mathrm{Tr}(A_k)\mathrm{Tr}(B_k)}{d_k} =
\langle \alpha,\beta\rangle k^{n+2} + O(k^{n+1}).\]
When $V$ is a smooth manifold, then this inner product can also be
computed differential geometrically. It was originally introduced in
this form by Futaki-Mabuchi~\cite{FM}.

To define the relative Futaki invariant, suppose that we have a torus
action $T$ on $(V,L)$ commuting with $\alpha$. Let us write
$\overline{\alpha}$ for the projection of $\alpha$ orthogonal to $T$,
with respect to the inner product we defined. Then we define 
the relative Futaki invariant $F_T(\alpha)$ by
\[ F_T(\alpha) = F(\overline{\alpha}).\]
Equivalently if $\beta_1,\ldots,\beta_d$ is a basis of
$\mathbf{C}^*$-actions generating the torus $T$, then 
\[ F_T(\alpha) = F(\alpha) - \sum_{i=1}^d \frac{\langle
\alpha,\beta_i\rangle}{\langle \beta_i,\beta_i\rangle} F(\beta_i).\]

It will be convenient for us to extend these definitions to 
$\mathbf{Q}$-line bundles using the relation
\[ F(V,L^r,\alpha) = r^n F(V,L,\alpha),\]
which the reader can readily verify. It will also be useful to allow
rational multiples of $\mathbf{C}^*$-actions. For this we use the relation
\[ F(V,L,r\alpha) = r F(V,L,\alpha).\]
We next recall the notion of a test-configuration from~\cite{Don02} with
the necessary modification for relative stability.

\begin{defn} 
  A \emph{test-configuration for $(X,L)$}
  consists of a $\mathbf{C}^*$-equivariant flat family of schemes
  $\pi:\mathcal{X}\to\mathbf{C}$ (where $\mathbf{C}^*$ acts on
  $\mathbf{C}$ by multiplication) and a $\mathbf{C}^*$-equivariant,
  relatively  ample
  $\mathbf{Q}$-line bundle $\mathcal{L}$ over 
  $\mathcal{X}$.  We require that the
  fibres $(\mathcal{X}_t,\mathcal{L}|_{\mathcal{X}_t})$ are isomorphic
  to $(X,L)$ for $t\not=0$, where $\mathcal{X}_t=\pi^{-1}(t)$. The
  test-configuration is called a \emph{product configuration} if
  $\mathcal{X} = X\times\mathbf{C}$. 

  We say that the test-configuration is \emph{compatible with a torus $T$
  of automorphisms of $(X,L)$}, if there is a torus action on
  $(\mathcal{X},\mathcal{L})$ which preserves the fibres of
  $\pi:\mathcal{X}\to\mathbf{C}$, commutes with the
  $\mathbf{C}^*$-action, and restricts to $T$ on
  $(\mathcal{X}_t,\mathcal{L}|_{\mathcal{X}_t})$ for $t\not=0$. 
\end{defn}

Note that given a test-configuration $(\mathcal{X},\mathcal{L})$, there
is an induced $\mathbf{C}^*$-action $\alpha$ on the central fibre
$(\mathcal{X}_0,\mathcal{L}|_{\mathcal{X}_0})$. We will write
$F(\mathcal{X},\mathcal{L})$ for the Futaki invariant of this induced
action $\alpha$.  
With these preliminaries we can state the main definition.

\begin{defn}\label{def:kstable}
  A polarised variety $(X,L)$ is \emph{K-semistable relative to a
  torus $T$ of automorphisms} if 
  $F_T(\mathcal{X},\mathcal{L})\geqslant 0$ for all
test-configurations compatible with the torus. If in addition 
equality holds only for the product configuration, then $(X,L)$ is
\emph{K-polystable} relative to the torus $T$. 
\end{defn}

If we have two tori $T'\subset T$ acting on $(X,L)$, then
K-polystability relative to $T$ is a weaker condition than relative to
$T'$, since there are fewer test-configurations compatible with a larger
torus. Thus, the weakest notion is K-polystability relative to a maximal
torus of automorphisms. The strongest notion is K-polystability relative
to the extremal $\mathbf{C}^*$-action. This is a $\mathbf{C}^*$-action
$\chi$ defined by Futaki-Mabuchi~\cite{FM} as follows. Fix a maximal
torus of automorphisms $T$, and write $\mathfrak{t}$ for its Lie
algebra. The Futaki invariant gives a linear map
$\mathfrak{t}\mapsto\mathbf{C}$, and $\chi$ is dual to this map under
the inner product on $\mathfrak{t}$. This gives a $\mathbf{C}^*$-action
on $(X,L)$, unique up to conjugation. In particular if the Futaki
invariant of any $\mathbf{C}^*$-action on $(X,L)$ vanishes, then
$\chi=0$, and K-polystability relative to $\chi$ is simply
K-polystability. 

It would be interesting to strengthen the conclusion of
Theorem~\ref{thm:main} to K-poly\-sta\-bi\-lity relative to the extremal
$\mathbf{C}^*$-action. Note that the analogous statement is true 
in finite dimensional geometric invariant theory, by Theorem 3.5 in
\cite{GSz04} (the same proof works if we replace the maximal torus with
any torus containing the extremal $\mathbf{C}^*$-action). 

We next recall the two theorems that we will use in the next
section.

\begin{thm}\label{thm:semistab}
	If $M$ admits an extremal metric in
	$c_1(L)$ then $(M,L)$ is K-semistable relative to a maximal
	torus of automorphisms.
\end{thm}
\begin{proof}
	This follows easily from Donaldson's lower bound for the Calabi
	functional \cite{Don05}. For details see \cite{GSzThesis}. For
	the convenience of the reader we outline the argument here.
	Donaldson's lower bound tells us that for any
	test-configuration, if $\alpha$ is the induced
	$\mathbf{C}^*$-action on the central fiber, then  
\begin{equation}\label{eq:skd}
	\inf_{\omega\in c_1(L)} c_n\Vert S(\omega) - \hat{S}\Vert_{L^2}
	\geqslant \frac{-F(\alpha)}{\Vert\alpha\Vert},
\end{equation}
where $c_n$ is a constant depending only on the dimension, 
$\Vert\alpha\Vert=\langle\alpha,\alpha\rangle^{1/2}$  using the
inner product defined above, and $\hat{S}$ is the average of the scalar
curvature $S(\omega)$. Moreover, if $\omega$ is an extremal metric, then 
\begin{equation}\label{eq:ext}
	c_n\Vert S(\omega) - \hat{S}\Vert_{L^2} = \frac{ F(\chi) }{\Vert
\chi\Vert} = \Vert\chi\Vert, 
\end{equation}
where $\chi$ is the extremal vector field on $(M,L)$. We are using here
that $F(\chi)=\langle\chi,\chi\rangle$ by definition of the extremal
vector field. It follows from (\ref{eq:skd}) and (\ref{eq:ext}) that if
$M$ admits an extremal metric in $c_1(L)$ then
\begin{equation} \label{eq:lower}
	\frac{F(\alpha)}{\Vert \alpha\Vert} \geqslant -\Vert\chi\Vert
\end{equation}
for all test-configurations.

Suppose now that $M$ admits an extremal metric in $c_1(L)$, and
we have a test-configuration for $(M,L)$ which is
compatible with a maximal torus of automorphisms $T$. Write $\alpha$ for
the induced $\mathbf{C}^*$-action on the central fiber. By twisting the
$\mathbf{C}^*$-action on the total space by the projection of $\alpha$
onto $T$ if necessary, we can assume that $\alpha$ is orthogonal to
$T$.
We want to show that $F(\alpha)\geqslant 0$. Suppose on the contrary
that $F(\alpha) < 0$, and let $\mu>0$ satisfy $F(\mu\alpha) =
-\Vert\mu\alpha\Vert^2$. By pulling back the test-configuration under a
base change $z\mapsto z^r$, and twisting the action on the total space by
the inverse of $\chi$, we obtain a test-configuration for $(M,L)$ such
that the action on the central fiber is $r(\mu\alpha-\chi)$, where $r$
is large enough to make this a genuine $\mathbf{C}^*$-action. 
From (\ref{eq:lower}) we know that 
\[  \frac{F(\mu\alpha - \chi)}{\Vert \mu\alpha-\chi\Vert} = 
\frac{F(r(\mu\alpha-\chi))}{\Vert r(\mu\alpha-\chi)\Vert} \geqslant
-\Vert\chi\Vert.\]
But at the same time
\[ F(\mu\alpha - \chi) = -\Vert\mu\alpha\Vert^2 - \Vert\chi\Vert^2
= -\Vert \mu\alpha - \chi\Vert^2,\]
since $\alpha$ is orthogonal to $\chi$. So 
\[ \frac{F(\mu\alpha - \chi)}{\Vert\mu\alpha - \chi\Vert} = -\Vert
\mu\alpha - \chi\Vert < -\Vert\chi\Vert.\]
This contradiction shows that $(M,L)$ is K-polystable relative to $T$.
The same argument also shows that $(M,L)$ is K-polystable relative to
the extremal $\mathbf{C}^*$-action.
\end{proof}

\begin{thm}[Arezzo-Pacard-Singer \cite{APS}]\label{thm:APS}
	Suppose that $M$ admits an
	extremal metric in $c_1(L)$, and let $T$ be a maximal torus of
	automorphisms of $(M,L)$. If $p\in M$ is a fixed point of
	$T$, then the blowup $\mathrm{Bl}_p M$ of $M$ at $p$ admits an
	extremal metric in the class $c_1(\pi^*L - \epsilon E)$ for
	sufficiently small $\epsilon>0$. Here $\pi$ is the blowdown map,
	and $E$ is the exceptional divisor.
\end{thm}
\begin{proof}
	This follows from \cite{APS} Theorem 2.1. Indeed we
	can choose an extremal metric $\omega$ on $M$ such that the
	isometry group of $\omega$ contains a compact maximal torus
	$T_\mathbf{R}$, which is contained in the
	complex torus $T$. In the notation of \cite{APS}
	we let $K=T_\mathbf{R}$,
	and let $\mathfrak{k}$ be its Lie algebra. Since $K$ is a
	maximal torus, any $K$-invariant holomorphic hamiltonian vector
	field lies in $\mathfrak{k}$. Moreover if we write $S(\omega)$ for the 
	scalar curvature then by Calabi's theorem~\cite{Cal85} the vector field
	$J\nabla S(\omega)$ lies in the center of the Lie algebra of Killing fields, so it also
	 lies in $\mathfrak{k}$.
	 This allows us to apply \cite{APS}
	Theorem 2.1, and we get the stated result.
\end{proof}

\section{Proof of Theorem \ref{thm:main}}
Let us suppose that $M$ admits an extremal metric in $c_1(L)$ and 
choose a maximal torus $T\subset \mathrm{Aut}(M,L)$. From
Theorem~\ref{thm:semistab} we know that 
if $(\mathcal{X},\mathcal{L})$ is a test-configuration for
$(M,L)$ compatible with $T$, then the relative Futaki invariant
satisfies
$F_T(\mathcal{X})\geqslant0$. Suppose
then that $F_T(\mathcal{X})=0$. 

We can assume that $M\subset\mathbf{P}(V)$, where $V=H^0(M,L)^*$.
Moreover the torus $T$ acts on $\mathbf{P}(V)$, preserving $M$. In addition there
is an extra $\mathbf{C}^*$-action $\alpha$ on $\mathbf{P}(V)$, commuting with the
$T$-action and such that the flat closure of the family $t\mapsto
\alpha(t)\cdot M$ across $t=0$ is the test-configuration $\mathcal{X}$.
Let us write $(M_0,L_0)$ for the central fiber of the
test-configuration. Then we have both $\alpha$ and the torus $T$ acting
on $(M_0,L_0)$. 
By twisting the action on the total space by the orthogonal projection
of $\alpha$ onto $T$ (which does not change the relative
Futaki invariant), we can assume that
$\langle\alpha, T\rangle=0$. In this case
\[ F_T(\mathcal{X},\mathcal{L}) = F(M_0,L_0,\alpha).\]
We now look at the weight decomposition under
$\alpha$ given by
\[ V = \bigoplus_i V_{m_i},\]
where $m_0 < m_1 <\ldots < m_L$ for some $L>0$, and consider the least
$l\geqslant 0$ such that 
\[ \mathrm{red}(M_0)\subset \mathbf{P}\big(\bigoplus_{i\leqslant l}
V_{m_i}\big).\]
It is proved in \cite{Sto08} section 3 that if $l=0$, so that $\alpha$
acts trivially on $\mathrm{red}(M_0)$, then
either $\mathcal{X}$ is a product test-configuration, or
$F(M_0,L_0,\alpha)>0$, which is a contradiction. On the other hand, if
$l>0$, then consider the repulsive fixed point set 
\[ M_0' = \mathrm{red}(M_0)\cap \mathbf{P}(V_{m_l}).\]
The set of points $p\in M$ for which the limit 
\[ q = \lim_{t\to 0} \alpha(t)p \]
is in $M_0'$ is precisely
\[ M' = M \cap \mathbf{P}\left(\bigoplus_{i\geqslant l} V_{m_i}\right).\]
This is a closed $T$-invariant set, so it contains a point $p$ fixed by
$T$. To see this, we can take a basis of $\mathbf{C}^*$-actions
$\beta_i$ generating the torus $T$, and then given any point $p$ in $M'$
we can inductively move it to a fixed point of $\beta_i$ by taking the
limit of $\beta_i(t)p$ as $t\to 0$. Doing this for each $i$, we end up with
a fixed point of $T$.
The corresponding limit $q$ will then be a $T$-invariant, repulsive
fixed point of $\alpha$ in $\mathrm{red}(M_0)$.

Letting $Z\subset \mathcal{X}$ 
be the closure of the orbit of $p$ under $\alpha$, we obtain a
test-configuration 
\begin{equation*}
(\widehat{\mathcal{X}}, \widehat{\mathcal{L}}) =
(\mathrm{Bl}_Z\mathcal{X}, \phi^*\mathcal{L} - \epsilon E)
\end{equation*}
for the polarised manifold $(\mathrm{Bl}_p M,
\phi^*L - \epsilon E)$, where $\phi\!: \widehat{\mathcal{X}}\to\mathcal{X}$ is the
blowdown. The only nontrivial thing to check is flatness of
the composition $\pi\circ \phi\!:\widehat{\mathcal{X}}\to \mathcal{X}\to
\mathbf{C}$. This holds because blowing up $Z \subset \mathcal{X}$ does not
introduce new associated points (i.e. embedded schemes) of
$\mathcal{X}$, only the Cartier exceptional divisor $E$ (for details see
the proof of Proposition 2.13 of \cite{Sto08}).     

For suitably small $\epsilon > 0$ the
test-configuration $(\widehat{\mathcal{X}}, \widehat{\mathcal{L}})$ will
have negative Futaki invariant, and in fact it
will even have negative Futaki invariant relative to $T$. This follows
from the lemma below and its corollary. 

At the same time from Theorem~\ref{thm:APS} we know that 
$\mathrm{Bl}_pM$ admits an extremal metric in the class $c_1(\phi^*L -
\epsilon E)$ for 
suitably small $\epsilon$ since $p$ is fixed by the torus $T$, which is
a maximal torus of automorphisms of $M$. This contradicts
Theorem~\ref{thm:semistab}, and completes the proof of the main theorem.

\begin{lem} Let $(\mathcal{X},\mathcal{L})$ 
	be a test-configuration for $(M,L)$
	compatible with a torus $T$ of automorphisms, and suppose
	that the induced action $\alpha$ on the central fiber satisfies
	$\langle\alpha,T\rangle = 0$. Let
	$\widehat{\mathcal{X}}$ be given by the blowup of a
	$T$-invariant section as described above. Then 
	\begin{equation}
		F(\widehat{\mathcal{X}},\widehat{\mathcal{L}}) = 
		F(\mathcal{X},\mathcal{L}) + \left(\lambda(q) -
		\frac{b_0}{a_0}\right)
	\frac{\epsilon^{n-1}}{2(n-2)!} + O(\epsilon^n),
	\end{equation}
	and
	\[\langle\hat{\alpha}, \hat{T}\rangle = O(\epsilon^n),\]
	where we use the $\mathbf{Q}$-polarization
	$\widehat{\mathcal{L}}=\phi^*\mathcal{L} - \epsilon E$
	on $\widehat{\mathcal{X}}$ for some small rational $\epsilon
	>0$, and $\hat{\alpha}, \hat{T}$ are the actions of
	$\alpha$ and $T$ lifted to the blowup. 
	It follows that the relative Futaki invariants
	satisfy
	\[
	F_{T}(\widehat{\mathcal{X}},\widehat{\mathcal{L}})
	= F_T(\mathcal{X},\mathcal{L}) +  
	\left(\lambda(q) - \frac{b_0}{a_0}\right)
	\frac{\epsilon^{n-1}}{2(n-2)!} +
	O(\epsilon^n).\]
		Here $\lambda(q)$ is the weight of $\alpha$ on the fiber
		$L_{0,q}$, 
	and $a_0, b_0$ are defined by the expansions of the dimension
	and weight on $H^0(M_0,L^k_0)$ calculated at the central fiber
	of $\mathcal{X}$ as usual:
	\[ \begin{aligned}
		d_k &= a_0k^n + a_1k^{n-1} + \ldots, \\
		w_k &= b_0k^{n+1} + b_1k^n + \ldots.
	\end{aligned}\]
\end{lem}
\begin{proof}
The central fibre of $\widehat{\mathcal{X}}$ will not in general be
isomorphic to $\widehat{M}_0 := \mathrm{Bl}_{q} M_0$. In fact it will
contain another large component $P$ glued to $\widehat{M}_0$ along the
exceptional divisor $E'$ for the morphism $\widehat{M}_0 \to
M_0$, as we now explain. 

By \cite{hart}, II Corollary 7.15, there is a closed immersion
$\widehat{M}_0 \hookrightarrow \widehat{\mathcal{X}}_0$ induced by the
closed immersion $M_0 \subset \mathcal{X}$ under blowing up $Z$. Let
$\mathcal{I}_q \subset \mathcal{O}_{M_0}$ denote the ideal sheaf of $q
\in M_0$. By the algebraic definition of blowing up we have
$\widehat{M}_0 \cong \mathrm{Proj} \bigoplus_{k \geq 0}
\mathcal{I}^k_q$. On the other hand the generic fibre of
$\widehat{\mathcal{X}}$ is $\mathrm{Proj} \bigoplus_{k \geq 0}
\mathcal{I}^k_p$, where $\mathcal{I}_p$ is the ideal of the smooth point
$p \in M$. Thus by the numerical criterion for flatness when the
Hilbert-Samuel polynomial for $p \in M$ is larger that that of $q \in
M_0$ (i.e. when $q$ is singular enough as a point of $M_0$) there will
be an additional component $P$ in the central fibre, given by the
closure of $\widehat{\mathcal{X}}_0\setminus\widehat{M}_0$. A simple
example has been suggested by S. Donaldson: when $q$ is an isolated
threefold ordinary double point inside the central fibre one has $P
\cong \mathbf{P}^3$ glued in along a smooth quadric. Note that this is
different from the situation described in \cite{jag} section 2, where
the central fibre of the original test configuration is smooth
(isomorphic to $M$), but one blows up $0-$cycles instead of just a point.
In any case the restriction $\widehat{\mathcal{L}}_{0|_{\widehat{M}_0}}$
is just $\phi^*L_0 - \epsilon E'$.

Taking this information into account we now compute the Donaldson-Futaki
invariant for the action $\alpha$ on the central fiber
$\widehat{\mathcal{X}}_0$. In the
calculations that follow $\epsilon$ is a fixed positive rational number,
and we tacitly restrict to those $k \gg 1$ for which $\epsilon k$ is an
integer. We also suppress pullbacks like $\pi^*$ or $\phi^*$ when this
causes no confusion. By flatness, using the Riemann-Roch theorem we have 
\begin{equation}\label{eq:RR}
	\begin{aligned}
h^0(\widehat{\mathcal{X}}_0, \widehat{\mathcal{L}}^k_0) &= h^0(\mathrm{Bl}_p M, L^k - k\epsilon E)\\
&= h^0(M, L^k) - \frac{\epsilon^n}{n!}k^n -\frac{\epsilon^{n-1}}{2(n-2)!}k^{n-1} + \ldots.
\end{aligned}
\end{equation}
Using the restriction $\mathbf{C}^*$-equivariant exact sequence 
\begin{equation}\label{eq:restrict}
0 \longrightarrow
H^0_P(\mathcal{I}^{k\epsilon}_{E'}\widehat{\mathcal{L}}^k_{0|_P})\longrightarrow
H^0_{\widehat{\mathcal{X}}_0}(\widehat{\mathcal{L}}^k_0) \longrightarrow
H^0_{\widehat{M}_0}(L^k_0 - k\epsilon E')\longrightarrow 0
\end{equation}
which holds for large $k \gg 1$, we find
\begin{equation*}
\mathrm{Tr}(H^0_{\widehat{\mathcal{X}}_0}(\widehat{\mathcal{L}}^k_0)) = \mathrm{Tr}(H^0_{\widehat{M}_0}(L^{k}_0 - k\epsilon E')) + \mathrm{Tr}(H^0_P(\mathcal{I}^r_{E'}\widehat{\mathcal{L}}^k_{0|_P})).
\end{equation*}
Note that $H^0_{\widehat{M}_0}(L^{k}_0 - k\epsilon E') \cong
H^0_{M_0}(\mathcal{I}^{k \epsilon}_{q} L^{k}_0)$ so the first term in
the formula above equals $\mathrm{Tr}(H^0_{M_0}(L^k_0))
-\mathrm{Tr}(H^0(\mathcal{O}_{k \epsilon q} \otimes L^{k}_0|_q))$.
From the exact sequence
\[ 0\longrightarrow \mathcal{I}_q^{k\epsilon}L_0^k\longrightarrow
L_0^k\longrightarrow \mathcal{O}_{k\epsilon q}\otimes
L_0^k|_q\longrightarrow 0,\]
together with (\ref{eq:RR}) and (\ref{eq:restrict}) we see that the
length of the $\mathcal{O}_{M_0}-$module
$\mathcal{O}_{k\epsilon q}$ is given by 
\[h^0_{P}(\mathcal{I}^{k\epsilon}_{E'}\widehat{\mathcal{L}}^k_{0|_P}) + \frac{\epsilon^n}{n!}k^n +
	\frac{\epsilon^{n-1}}{2(n-2)!}k^{n-1} + O(k^{n-2}).\]
It follows that the weight of the action on $\mathcal{O}_{k\epsilon q}\otimes L^k|_q$ is given by 
    \[ \begin{aligned}
		w(\mathcal{O}_{k\epsilon q}\otimes L^k|_p) &= w(\mathcal
	{O}_{k\epsilon q}) + k\lambda(q)\mathrm{len}(\mathcal{O}_{k\epsilon
	q}) \\
	&= k\lambda(q)h^0_{P}(\mathcal{I}^{k\epsilon}_{E'}\widehat{\mathcal{L}}^k_{0|_P}) \\&+ \left(c_0\epsilon^{n+1} + \lambda(q)
	\frac{\epsilon^n}{n!}\right)k^{n+1} +
	\left(c_1\epsilon^n +
	\lambda(q)\frac{\epsilon^{n-1}}{2(n-2)!}\right)k^n + \ldots,
	\end{aligned}\]
where $c_0, c_1$ are given by the expansion 
	\[ w(\mathcal{O}_{k\epsilon q}) = c_0(k\epsilon)^{n+1} +
	c_1(k\epsilon)^n +
	\ldots.\]
Similarly $\mathcal{I}^{k\epsilon}_{E'}\widehat{\mathcal{L}}^k_{0|_P}
\cong \mathcal{L}_0^{k}|_q \otimes
\mathcal{I}^{k\epsilon}_{E'}\mathcal{O}(-k E)|_P$ and the action on the
latter factor has vanishing weight, so one has
\[\mathrm{Tr}(H^0_{P}(\mathcal{I}^{k\epsilon}_{E'}\widehat{\mathcal{L}}^k_{0|_P})) = k\lambda(q)h^0_{P}(\mathcal{I}^{k\epsilon}_{E'}\widehat{\mathcal{L}}^k_{0|_P}).\]
After a simple cancellation we find
\[ \begin{aligned}
		\hat{a}_0 &= a_0 + O(\epsilon^n) \\
		\hat{a}_1 &= a_1 - \frac{\epsilon^{n-1}}{2(n-2)!} \\
		\hat{b}_0 &= b_0 + O(\epsilon^n) \\
		\hat{b}_1 &= b_1 -
		\lambda(q)\frac{\epsilon^{n-1}}{2(n-2)!},
	\end{aligned}\]
where $\hat{a}_i, \hat{b}_i$ are computed on $\widehat{\mathcal{X}}$. Using the formula $F(\widehat{\mathcal{X}}) = \displaystyle{\frac{\hat{a}_1}{\hat{a}_0}}\hat{b}_0
	- \hat{b}_1$ we get
\begin{equation*}
F(\widehat{\mathcal{X}}) = F(\mathcal{X}) + \left(\lambda(q) - \frac{b_0}{a_0}\right)
	\frac{\epsilon^{n-1}}{2(n-2)!} + O(\epsilon^n).
\end{equation*}	

Now let $\beta$ be any $\mathbf{C}^*$-action in the torus $T$.
To compute the inner product $\langle \hat{\alpha}, \hat{\beta}\rangle$,
let us write $A_k, B_k$ for the infinitesimal generators of the actions
$\alpha,\beta$ on $H^0(M_0, L^k_0)$, and write $\hat{A}_k,
\hat{B}_k$ for the
	infinitesimal actions of the corresponding actions on
	$H^0(\widehat{\mathcal{X}}_0, \widehat{\mathcal{L}}^k_0)$. The inner product
	$\langle\hat{\alpha},\hat{\beta}\rangle$ is the leading order
	term in 
	\begin{equation}\label{product}
		\mathrm{Tr}(\hat{A}_k\hat{B}_k) -
		\frac{\mathrm{Tr}(\hat{A}_k) 
	\mathrm{Tr}(\hat{B}_k)}{\hat{d}_k}.
	\end{equation}
Since the actions $\alpha, \beta$ commute, we can use precisely the same
exact sequences as before to compute
\begin{align*} \mathrm{Tr}(A_kB_k) - \mathrm{Tr}(\hat{A}_k\hat{B}_k) &=
\lambda_{\alpha}(p)\lambda_{\beta}(p)(\mathrm{len}(\mathcal{O}_{
k\epsilon p}) - h^0_{P}(\mathcal{I}^{k\epsilon}_{E'}\widehat{\mathcal{L}}^k_{0|_P})) + d_0 k^{n+2}\\ &+
	\mathrm{Tr}(A'_{k\epsilon}B'_{k\epsilon}),
\end{align*}
	where $A'_{k\epsilon}$ and $B'_{k\epsilon}$ 
	are the infinitesimal generators of the
	actions $\alpha,\beta$ on $\mathcal{O}_{k\epsilon p}$. We have
	an expansion
	\[ \mathrm{Tr}(A'_{k\epsilon}B'_{k\epsilon}) = c_0' (\epsilon
	k)^{n+2} + O(k^{n+1}). \]
	So up to terms of order $\epsilon^n$, the 
	leading order term in \eqref{product} is the same as that
	in 
	\[ \mathrm{Tr}(A_kB_k) -
	\frac{\mathrm{Tr}(A_k)\mathrm{Tr}(B_k)}{ d_k},\]
	which is just $\langle\alpha,\beta\rangle = 0$. This shows
	$\langle\hat{\alpha},\hat{\beta}\rangle = O(\epsilon^n)$.
	A similar computation of the inner product on the blowup is in
	\cite{DV}.

	The statement about the relative Futaki invariants now
	follows from the definition
	\[ F_T(\widehat{\mathcal{X}},\hat{\alpha}) =
	F(\widehat{\mathcal{X}},\hat{\alpha}) - \sum_{i=1}^d\frac{\langle
	\hat{\alpha},\hat{\beta}_i\rangle}{\langle\hat{\beta}_i,
	\hat{\beta}_i\rangle}
	F(\widehat{\mathcal{X}},\hat{\beta}_i),\]
	where the $\mathbf{C}^*$-actions $\beta_i$ generate the torus
	$T$.
\end{proof}
\begin{cor} Following the notation above, if $q \in M_0$ is a repulsive fixed point for $\alpha$ then $F(\widehat{\mathcal{X}}) < F(\mathcal{X})$ for $\epsilon$ small enough.
\end{cor}
\begin{proof} It remains to prove that the highest order correction term
	\[\left(\lambda(q) - \frac{b_0}{a_0}\right)
	\frac{\epsilon^{n-1}}{2(n-2)!}\]
	 is negative. It is proved is
	\cite{jag} section 4 that, possibly after a fixed basechange of the test-configuration,
	the coefficient $\lambda(q) - \frac{b_0}{a_0}$ is integral and
	equals minus the Hilbert-Mumford weight of $q$ under the
	induced action of $\alpha$ on $\mathbf{P}(V)$. The Hilbert-Mumford criterion combined with a local computation then shows
	that the weight of such a repulsive fixed point
	must be positive (for details see the proof of Theorem 1.2 in
	\cite{Sto08}). 
	
	Alternatively we can give a self-contained proof as follows. 
	Let $M_0$ be the central fiber of our test-configuration and suppose
that $q$ is a repulsive fixed point with weight $m_l$ and also let $r$
be a point in $\mathrm{red}(M_0)\cap \mathbf{P}(V_{m_0})$, ie. a lowest weight
invariant point. Then as in the Futaki invariant calculation we have the
exact sequence
\[ 0\longrightarrow \mathcal{I}^{k\epsilon}_{r} L^k_0 \longrightarrow L^k_0
\longrightarrow \mathcal{O}_{\epsilon k r}\otimes L^k_0|_r \longrightarrow
0.\]
Write $-\lambda$ for the weight $m_l$, so $\lambda(q)=\lambda$ and
 $m_0\leqslant-\lambda-1$. 
 The weights on $L_0$ are the opposite by duality and they are all at least $\lambda$. 
 Using the notation from the proof of Theorem \ref{thm:main}, 
 from the exact sequence we have
\begin{equation}
	\label{eq:weight}\begin{aligned}
	w_k &= w(\mathcal{I}^{k\epsilon}_{r}L^k) +
	w(\mathcal{O}_{\epsilon k r})  -km_0\,\mathrm{len}(\mathcal{O}_{\epsilon k
	r}) \\
	&\geqslant k\lambda \big(d_k - \mathrm{len}(\mathcal{O}_{\epsilon kr})\big) +
	w(\mathcal{O}_{\epsilon k r}) +
	k(\lambda+1)\mathrm{len}(\mathcal{O}_{\epsilon kr}) \\
	&= k\lambda d_k + k\mathrm{len}(\mathcal{O}_{\epsilon kr}) +
	w(\mathcal{O}_{\epsilon kr}).
\end{aligned}\end{equation}
Now we need the expansions
\[\begin{gathered}
	\mathrm{len}(\mathcal{O}_{\epsilon kr}) = c(\epsilon k)^n + O(k^{n-1}) \\
	w(\mathcal{O}_{\epsilon kr}) = c'(\epsilon k)^{n+1} + O(k^n).
\end{gathered}\]
It is important here that $c > 0$. This follows from \cite{hart}, III Corollary 9.6. Then
looking at the $k^{n+1}$ term in (\ref{eq:weight}) we get
\[ b_0 \geqslant \lambda a_0 + c\epsilon^n + c'\epsilon^{n+1}. \]
When $\epsilon$ is chosen sufficiently small we get the required inequality
$ \frac{b_0}{a_0} > \lambda$.
\end{proof}

\end{document}